\documentclass[]{IEEEtran4PSCC}
\newcommand{\amdq}{Dell PowerEdge R415 servers with Dual 2.8GHz AMD 6-Core Opteron 4184 CPUs and 64GB of memory}

\usepackage{booktabs}

\usepackage{epsfig}
\usepackage{float, multi row, amsmath,mathabx}
\floatstyle{ruled}
\newfloat{model}{thp}{lop}
\floatname{model}{Model}
\usepackage{color}

\usepackage{amsthm,bm}
\usepackage[bookmarks=false]{hyperref}
\hypersetup{
colorlinks=true, 
breaklinks=true,
urlcolor= blue, 
linkcolor= blue, 
citecolor=red, 
}
\newtheorem{theorem}{Theorem}

\newtheorem{remark}{Remark}
\newtheorem*{remark*}{Remark}
\newtheorem{definition}{Definition}
\ifCLASSINFOpdf
\else
\fi
\begin{document}
\title{Polynomial SDP Cuts for Optimal Power Flow}

% To specify the authors when (number of affiliations > 2)
 \author{\IEEEauthorblockN{Hassan Hijazi\IEEEauthorrefmark{1}\IEEEauthorrefmark{2},
Carleton Coffrin\IEEEauthorrefmark{2} and
Pascal Van Hentenryck\IEEEauthorrefmark{2}\IEEEauthorrefmark{3}}
 \IEEEauthorblockA{\IEEEauthorrefmark{1} The Australian National University, Canberra, Australia \{hassan.hijazi@anu.edu.au\}}
 \IEEEauthorblockA{\IEEEauthorrefmark{2} National ICT Australia\\
}
 \IEEEauthorblockA{\IEEEauthorrefmark{3} University of Michigan, Ann Arbor, MI, USA.\\
}

}

\maketitle

\begin{abstract}
  The use of convex relaxations has lately gained considerable
  interest in Power Systems.  These relaxations play a major role in
  providing quality guarantees for non-convex optimization problems.
  For the Optimal Power Flow (OPF) problem, the semidefinite
  programming (SDP) relaxation is known to produce tight lower bounds.
  Unfortunately, SDP solvers still suffer from a lack of scalability.
  In this work, we introduce an exact reformulation of the SDP relaxation, 
  formed by a set of polynomial constraints defined in the space of real variables. 
  The new constraints can be seen as ``cuts", strengthening weaker second-order cone relaxations, 
  and can be generated in a lazy iterative fashion.
  The new formulation can be handled by standard nonlinear programming
  solvers, enjoying better stability and computational efficiency.
  This new approach benefits from recent results on
  tree-decomposition methods, reducing the dimension of the underlying
  SDP matrices. As a side result, we present a formulation of
  Kirchhoff's Voltage Law in the SDP space and reveal the existing
  link between these cycle constraints and the original SDP relaxation
  for three dimensional matrices.  Preliminary results show a
  significant gain in computational efficiency compared to a standard
  SDP solver approach.
\end{abstract}


% Use this to place sponsorships

\section*{Nomenclature}
\addcontentsline{toc}{section}{Nomenclature}
\begin{IEEEdescription}[\IEEEusemathlabelsep\IEEEsetlabelwidth{$\underline{\bm p}^g_{i}, \overline{\bm p}^g_{i}$}]
\item [{$N$}] The set of nodes in the network.
\item [{$E$}] The set of directed edges in the network.
\item [{$E^r$}] The set of edges in $E$ with reversed direction.
\item [{$\bm i$}] Imaginary number constant.
\item[$V_i$] Complex voltage at node $i$.
\item[$v_i$] Voltage magnitude at node $i$.
\item[$w_i$] Voltage magnitude squared at node $i$.
%\item [$w_i$] Square of voltage magnitude at node $i$.
%item [$l_{ij}$] Square of current magnitude along line $(i,j)$.
\item [$\theta_{i}$] Phase angle at node $i$.
\item [$\theta_{ij}$] Phase angle difference along line $(i,j)$.
%\item $q^g_{i}:$ Reactive power generation at node $i$.
\item [$S_{ij}$] Complex power flow along line $(i,j)$.
\item [$p_{ij},q_{ij}$] Active/reactive power flow along line $(i,j)$.
\item [$p^g_{i},q^g_{i}$] Active/reactive power generation at node $i$.
%\item[$\bm i$] imaginary number.
%\item[ $q_{ij}:$ Reactive power flow along line $(i,j)$.
\item [${\bm r_{ij}}, {\bm x_{ij}}$] Line resistance/reactance along line $(i,j)$.
\item [${\bm t_{ij}}$] Thermal limit along line $(i,j)$.
\item [$\underline{\bm \theta}_{ij}, \overline{\bm \theta}_{ij}$] Angle difference along line $(i,j)$.
\item [$\underline{\bm v}_{i}, \overline{\bm v}_{i}$] Voltage magnitude bounds at node $i$.
\item [$\underline{\bm p}^g_{i}, \overline{\bm p}^g_{i}$] Active power generation bounds at node $i$.
\item [$\underline{\bm q}^g_{i}, \overline{\bm q}^g_{i}$] Reactive power generation bounds at node $i$.
\item [$\bm c_{i}, \bm c^{'}_{i}$] Generation coefficient costs at node $i$.
\item [${\bm p^d_{i}}, {\bm q}^d_{i}$] Active/reactive power demand at node $i$.
\end{IEEEdescription}

\section{Introduction}

Power systems operations, design and planning problems rely on the set
of non-convex AC power flow equations.  Solution quality guarantees
have driven researchers to derive convex relaxations for various
problems in power systems such as the Optimal Power Flow (OPF)
%\cite{Jabr_06,farivar:etal:11,Li:12,Sojoudi_12,farivar:low:13,lavaei:low:12,sojoudi:Lavaei:12,Molzahn_13,Molzahn_14,Lavaei_14,Hij_QC_14}.
\cite{Jabr_06}-\cite{Hij_QC_14}.
Several results focus on demonstrating the tightness of these
relaxations under various assumptions. For the OPF problem, the semidefinite
  programming (SDP) relaxation is known to produce tight lower bounds.
  Unfortunately, SDP solvers still suffer from a lack of scalability.
  In this work, we introduce an exact reformulation of the SDP relaxation, 
  formed by a set of polynomial constraints defined in the space of real variables. 
  The new constraints can be seen as ``cuts", strengthening weaker second-order cone relaxations.
  
The rest of the paper is organized as follows. 
Section \ref{sec:rel} introduces the second-order cone and the SDP relaxations for the OPF problem. Section \ref{sec:sdp} presents the new polynomial SDP cuts and Section \ref{sec:cycle} investigates an original representation of cycle constraints and their link with the SDP relaxation. Finally, some preliminary results are presented in Section \ref{sec:numerical}.

%Assume that $\overline{\bm v}_{i} -  \underline{\bm v}_{i} \le 10\%$
\section{Power Flow Models and Relaxations} \label{sec:rel}

This section reviews the AC power flow equations and some of their 
relaxations. All power flow models share a set of operational bounds
constraints, i.e., 
\begin{align}
&p^2_{ij} + q^2_{ij} \leq {\bm t_{ij}},\forall (i,j),(j,i) \in E,\label{eq:op:capacity}\\
&\underline{\bm p}^g_{i} \leq p^g_{i} \leq \overline{\bm p}^g_{i}, \forall i\in N, \label{eq:op:pgen} \\
&\underline{\bm q}^g_{i} \leq q^g_{i} \leq \overline{\bm q}^g_{i}, \forall i\in N, \label{eq:op:qgen}\\
&\underline{\bm v}_{i} \leq v_{i} \leq \overline{\bm v}_{i}, \forall i\in N, \label{eq:op:volt}\\
&\underline{\bm \theta}_{ij} \leq \theta_{ij} \leq \overline{\bm \theta}_{ij}, \forall (i,j) \in E, \label{eq:op:diff}
\end{align}
Kirchhoff's Current Law, i.e.,
\begin{align}
& p^g_i - {\bm p^d_i}   = \sum_{\substack{(i,j)\in E}} p_{ij}, \forall i\in N, \label{eq:pcons}\\ 
& q^g_i - {\bm q^d_i} = \sum_{\substack{(i,j)\in E}} q_{ij}, \forall i\in N. \label{eq:qcons}
\end{align}
and the power equations, i.e.,
%
%\begin{align}
%&p_{ij} + i q_{ij} = (\boldsymbol g_i - \boldsymbol i \boldsymbol b_i)(V_iV^*_i - V_iV^*_j)
%\end{align}
%Expanding the constants and collecting in terms of $p_{ij}$ and $q_{ij}$ yields,
%\begin{align}
%&p_{ij} = \boldsymbol g_i \Re(V_iV^*_i)  -  \boldsymbol g_i \Re(V_iV^*_j) -  \boldsymbol b_i \Im(V_iV^*_j) \\
%&q_{ij} = - \boldsymbol b_i \Re(V_iV^*_i)  +  \boldsymbol b_i \Re(V_iV^*_j) -  \boldsymbol g_i \Im(V_iV^*_j)
%\end{align}
\begin{align}
&p_{ij} = \boldsymbol g_{ij} v_i^2  -  \boldsymbol g_{ij} {v_i}{v_j}\cos(\theta_{ij}) -  \boldsymbol b_{ij} {v_i}{v_j}\sin(\theta_{ij}) \label{eq:ac:p}\\
&q_{ij} = - \boldsymbol b_{ij} v_i^2  +  \boldsymbol b_{ij} {v_i}{v_j}\cos(\theta_{ij}) -  \boldsymbol g_{ij} {v_i}{v_j}\sin(\theta_{ij})\label{eq:ac:q}
\end{align}
where ${\bm g}  = {\bm r}/({\bm r}^2 + {\bm x}^2)$ and ${\bm b}  = -{\bm x}/({\bm r}^2 + {\bm x}^2)$.

Using complex numbers, one can derive a compact representation of
\eqref{eq:ac:p}-\eqref{eq:ac:q} Let $S_{ij} = p_{ij} + \bm i q_{ij}$,
$V_i = v_i (\cos(\theta_i) + \bm i \sin(\theta_i))$, and $\boldsymbol
Y_{ij} = \boldsymbol g_{ij} + \bm i \boldsymbol b_{ij}$. Equations
\eqref{eq:ac:p}-\eqref{eq:ac:q} can be written:
\begin{align}
&S_{ij} = \boldsymbol Y^*_{ij} (V_iV_i^* - V_iV_j^*). \label{eq:ac:S}
\end{align}
Consider the complex variable product 
$$V_iV_j^* = {v_i}{v_j}(\cos(\theta_{ij}) + i\sin(\theta_{ij}))$$
and let 
\begin{equation}\label{eq:lift}
W_{ij} = {w}^R_{ij} + \bm i {w}^I_{ij} = V_iV_j^*
\end{equation}
where
\begin{align}
%& \tilde{w}_{ii} = v_i^2 \\
& {w}^R_{ij} = {v_i}{v_j}\cos(\theta_{ij}) \\
& {w}^I_{ij} = {v_i}{v_j}\sin(\theta_{ij})
\end{align}
Note that  $W_{ji} =  W_{ij}^* = {w}^R_{ij} - \bm i {w}^I_{ij}$ and $W_{ii} = v_i^2$.

In a similar fashion, let
\begin{equation}\label{eq:lift1}
w_{i} = v_i^2
\end{equation}
Given these variable substitutions, equations \eqref{eq:ac:p}-\eqref{eq:ac:q} become linear:
\begin{align}
&p_{ij} = \boldsymbol g_{ij} {w}_{i}  -  \boldsymbol g_{ij} {w}^R_{ij} -  \boldsymbol b_{ij} {w}^I_{ij}  \label{eq:op:pflow:w} \\
&q_{ij} = - \boldsymbol b_{ij} {w}_{i} +  \boldsymbol b_{ij} {w}^R_{ij} -  \boldsymbol g_{ij} {w}^I_{ij} \label{eq:op:qflow:w}
\end{align}
Consider the $|N| \times |N|$ Hermitian matrix defined as: 
\begin{equation}\label{hermit}
{W} = \left\{
\begingroup
\begin{array}{l@{}l}
%\addtolength{\jot}{0.4em}
%\tilde{w}_{ii} = \tilde{x}_i\tilde{x}_i^* = v_i^2,~\forall i \in N,\\\
{W}_{ii} =  w_i,~ \forall i \in N,\\[1mm]
{W}_{ij} = {w}^R_{ij} + \bm i {w}^I_{ij},~ \forall (i,j) \in E,\\[1mm]
{W}_{ji} = {w}^R_{ij} - \bm i {w}^I_{ij},~ \forall (i,j) \in E,\\[1mm]
{W}_{ij} = 0,~\forall (i,j) \notin E.
\end{array}
\endgroup
\right\}
\end{equation}
${W}$ is characterized by the following constraints:
 \begin{align}
{W} & \succeq 0 \label{eq:sdp}\\
\mbox{rank}({W}) & = 1\label{eq:rank}
\end{align}

%\newpage
\subsection{The SDP Relaxation}

Given a convex objective function, the semidefinite programming
relaxation outlined in Model \ref{mod:sdp} is obtained by discarding
the non-convex rank constraints \eqref{eq:rank}. Note that the voltage and phase angle bound constraints \eqref{eq:op:volt} and \eqref{eq:op:diff} can be represented in the $W$-space as follows,
\begin{align}
&\underline{\bm v}_{i}^2 \leq w_{i} \leq \overline{\bm v}_{i}^2, \forall i\in N,\label{eq:wbound}\\
&tan(\underline{\bm \theta}_{ij})w^R_{ij} \le w^I_{ij} \le tan(\overline{\bm \theta}_{ij})w^R_{ij},~ \forall (i,j) \in E.\label{eq:tbound}
\end{align}¥

\begin{model}
\small
\caption{The SDP Relaxation \label{mod:sdp}}

\addtolength{\jot}{0.2em}
\begingroup
\begin{align*}
\min~ & \sum_{i \in N} {\bm c}_ip^g_i + {\bm c}^{'}_i\left(p^g_i\right)^2\\
&p_{ij} = \boldsymbol g_{ij} {w}_{i}  -  \boldsymbol g_{ij} {w}^R_{ij} -  \boldsymbol b_{ij} {w}^I_{ij},~ \forall (i,j) \in E,\\
&q_{ij} = - \boldsymbol b_{ij} {w}_{i} +  \boldsymbol b_{ij} {w}^R_{ij} -  \boldsymbol g_{ij} {w}^I_{ij},~ \forall (i,j) \in E,\\
&{W} \succeq 0,\\
&(\ref{eq:op:capacity})-(\ref{eq:op:qgen}),~(\ref{eq:pcons})-(\ref{eq:qcons}),~(\ref{eq:wbound})-(\ref{eq:tbound}).
\end{align*}
\endgroup
\end{model}

%This model represents a convex semidefinite-Program (SDP), and can be solved using appropriate SDP solvers.
%Unfortunately, SDP solving technology strongly suffers from the curse of dimensionality. Even though the problem is convex, and optimality guarantees can be provided, computational cost can become prohibitive for reasonable size instances.

\subsection{The SOCP Relaxation}

Sojoudi and Lavaei \cite{Sojoudi_12} observed that the SDP model can
be further relaxed by posting the positive semidefinite constraints
on the $2 \times 2$ sub-matrices of W related to each line in the
network.
\begin{align}\label{mat:socp}
 \begin{bmatrix} 
{W}_{ii} &  {W}_{ij}\\ 
{W}_{ji} &  {W}_{jj}
\end{bmatrix}
\succeq 0,  ~ \forall (i,j) \in E
\end{align}
Following the characterization of a positive semidefinite matrix based on the properties of its principal minors, each constraint $(i,j) \in E$ from \eqref{mat:socp} is equivalent to the following set of second-order cone constraints:
\begin{align}
& {w}_{i} \ge 0,~{w}_{j} \ge 0, \label{eq:socp:pos}\\
& {w}_{i}{w}_{j} \geq  \left({w}^R_{ij}\right)^2 + \left({w}_{ij}^I\right)^2. \label{eq:socp:det}
\end{align}
This formulation was originally proposed by Jabr in \cite{Jabr_06}.
On acyclic networks, the systems of Equations
\eqref{eq:socp:pos}--\eqref{eq:socp:det} for all $(i,j) \in E$ is
strictly equivalent to constraint \eqref{eq:sdp} \cite{Sojoudi_12}.
In the presence of cycles, this relaxation can be weak, since the SOCP
relaxation can be thought of as the introducion {\em virtual phase
  shifters} \cite{Sojoudi_12} in the network to counteract the effects
of Ohm's Law in cycles.

% this is shown in Section \ref{sec:num}. 
%This will be explained in the next section defining cycle constraints.
\section{SDP-Determinant Cuts}
\label{sec:sdp}

This section defines SDP-Determinant cuts and an alternative
formulation of Model \ref{mod:sdp}.

\subsection{SDP Principal Minors Characterization}

A necessary and sufficient condition for a symmetric $n \times n$
matrix to be positive semidefinite is that all its principal minors
are positive \cite{Prussing86}.  A principal minor is the determinant
of the submatrix formed by deleting the $n-k$ rows and the corresponding $n-k$ columns
($1 \leq k \leq n$) from the original
matrix. For a given $k$, there are $n \choose k$ such submatrices
where $n \choose k$ denotes the binomial coefficient
$\frac{n!}{(n-k)!k!}$. Hence the total number of principal minors is
given by
$$m = \sum\limits_{k=1}^{n} {n \choose k} = 2^n - 1.$$ 
For a $3 \times 3$ matrix, there are $2^3 - 1 = 7$ submatrices to
consider, six of which correspond to constraints \eqref{eq:socp:pos}
and \eqref{eq:socp:det}. As a result, the only additional constraint
imposes that the determinant of the full matrix must be positive.

\subsection{SDP Determinant Cuts}

Based on the principal minors characterization, the SDP constraint
\eqref{eq:sdp} is equivalent to the set of polynomial
constraints derived from the $m$ principal minors:
 \begin{equation}\label{eq:det}
 \det \left({W}_i\right) \ge 0,~\forall i \in \{1,\dots,2^n-1\}
\end{equation}¥
where ${W}_i$ is the $i$-th submatrix that corresponds to a principal minor. 
Model \ref{mod:sdp} is thus equivalent to 

\begin{model}
\small
\caption{The SDP Determinant Cut Relaxation \label{mod:sdpc}}

\addtolength{\jot}{0.2em}
\begingroup
\begin{align*}
\min~ & \sum_{i \in N} {\bm c}_ip^g_i + {\bm c}^{'}_i\left(p^g_i\right)^2\\
&p_{ij} = \boldsymbol g_{ij} {w}_{i}  -  \boldsymbol g_{ij} {w}^R_{ij} -  \boldsymbol b_{ij} {w}^I_{ij},\\
&q_{ij} = - \boldsymbol b_{ij} {w}_{i} +  \boldsymbol b_{ij} {w}^R_{ij} -  \boldsymbol g_{ij} {w}^I_{ij},\\
& \det \left({W}_i\right) \ge 0,~\forall i \in \{1,\dots,2^n-1\} \\
&(\ref{eq:op:capacity})-(\ref{eq:op:qgen}),(\ref{eq:pcons})-(\ref{eq:qcons}),~(\ref{eq:wbound})-(\ref{eq:tbound}).
\end{align*}
\endgroup
\end{model}
%\newpage
\begin{theorem}\label{th:eq}
%The SDP relaxation defined in Model \ref{mod:sdp} is strictly equivalent to the SDP-Determinant cut formulation defined in Model \ref{mod:sdpc}.
Model \ref{mod:sdp} is strictly equivalent to Model \ref{mod:sdpc}.
\end{theorem}¥
\begin{proof}
This is based on the principal minor characterization of positive semidefinite matrices.
\end{proof}
For illustration purposes, Theorem \ref{th:eq} implies that the $3\times 3$ positive semidefinite condition,
 \begin{equation}\label{hermit3}
%{W} = \left\{
{W} =
 %\begin{align}\label{mat:sdp3}
 \begin{bmatrix} 
{W}_{11} &  {W}_{12} &  {W}_{13}\\ 
{W}_{12}^* &  {W}_{22} &  {W}_{23}\\ 
{W}_{13}^* &  {W}_{23}^* &  {W}_{33}
\end{bmatrix}
\succeq 0
%\end{align}
\end{equation}
\noindent
is strictly equivalent to the system of constraints,
$$\begin{cases}\label{det3}
(\ref{eq:socp:pos})-(\ref{eq:socp:det}),~ (i,j) \in \{(1,2),(1,3),(2,3)\},\nonumber\\
\det({W} ) \ge 0
\end{cases}$$

Given the strict equivalence highlighted in Theorem \ref{th:eq} and since the SDP constraint
\eqref{eq:sdp} defines a convex feasibility set, it follows that the
determinant constraints also define a convex feasible region.
Obviously, the number of these constraints is exponential in the size of the matrix.
Nevertheless, since a submatrix of a positive
semidefinite matrix also needs to be positive semidefinite, one can
add a subset of these exponentially many constraints and still define
a convex region. Moreover, a lazy constraint generation approach can be implemented here.
\begin{remark}
  Given a submatrix of dimension $n$, the convexity of the feasible
  region is only guaranteed if all determinant constraints of lower
  dimensions are also added simultaneously, i.e., $1,\dots,n-1$. For
  instance, one needs to include the SOCP constraints
  (\ref{eq:socp:pos})-(\ref{eq:socp:det}) before including
  three-dimensional determinant constraints.
\end{remark}

\subsection{Non-convex representation of a convex region}
Lasserre \cite{Lasserre10,Lasserre11} nicely pointed out that, 
under a mild nondegeneracy condition, if the feasible region is convex, 
even though its algebraic representation is not, one can still guarantee 
convergence to a global minimizer when using a logarithmic barrier function.

\begin{remark}
  Since the lower bound on the voltage magnitude is strictly positive,
  we have $W_{ii} > 0 (i \in N)$ which is sufficient to
  guarantee that the gradient of these determinant constraints is
  non-zero, thus meeting the nondegeneracy condition stated in
  \cite{Lasserre10,Lasserre11}.
\end{remark}
This result is very powerful in the current framework, as we have proved that 
the feasible region is convex, and the non-degeneracy condition applies, 
thus Model \ref{mod:sdpc} can be solved to global optimality using open-source
 state-of-the-art interior point algorithms which implement a logarithmic barrier function, e.g. Ipopt \cite{Ipopt}.

\subsection{Tree Decomposition}

Recently, Madani et al \cite{Madani_14} showed that one can replace
the high-dimensional SDP constraint \eqref{eq:sdp} by a set of
low-dimensional SDP constraints based on a tree-decomposition of the
power network. Let $\mathcal T = \left\{\mathcal N_t, \mathcal
  E_t\right\}$ denote such a decomposition. A node $n_t \in \mathcal
N_t$ corresponds to a \emph{bag} of nodes $\mathcal B_i \subseteq
\mathcal N$ in the original network. The main result can be stated as
follows,
\begin{theorem}[\cite{Madani_14}]
\begin{equation}
W \succeq 0 \equiv W_{\mathcal B_i} \succeq 0,~\forall \mathcal B_i \in \mathcal N_t
\end{equation}
\end{theorem}
\noindent
The tree-width of a graph is equal to the cardinality of the biggest
bag minus one.  The authors demonstrate that standard benchmarks have
a low tree-width and thus the dimension of the underlying SDP
matrices can be reduced accordingly, e.g., the tree width on the IEEE
118 bus benchmark is 4. This property should also hold for real-world
grids that tend to be sparse due to the high cost of line
installations.

This nice result leverages the efficiency of Model \ref{mod:sdpc}, as the 
number of determinant cuts to be generated reduces dramatically.

\section{Cycle Constraints}\label{sec:cycle}

This section establishes a connection between the SDP relaxation and
Kirchhoff's voltage law, which can be
viewed as a set of cycle-based equations:
\begin{equation}
\sum_{\substack{(i,j)\in \mathcal C}} (V_i - V_j) = 0,~ \forall \mathcal C \in \mathcal G.\label{eq:cycle}
\end{equation}
\noindent
where $\mathcal C$ denotes a cycle and $\mathcal G$ is a collection of
cycles forming a cycles basis of the graph $(N,E)$.
%Kocuk et al. \cite{Kocuk_15} recently proposed a representation of these cycle constraints in the $\left({w}_{i},~{w}_{j},~{w}^R_{ij},~{w}_{ij}^I\right)$ space, for three and four dimensional cycles. 
Kocuk et al. \cite{Kocuk_15} recently proposed a representation of
these cycle constraints in the $W$ space for three and four
dimensional cycles.  For higher dimensional cycles, the authors
propose to introduce artificial edges and corresponding variables,
leading to three and four-cycle decompositions of the cycle basis.  In
what follows, we present a new formulation of the cycle constraints
projected in the $W$ space, which applies for any cycle dimension and
does not require artificial edges/variables.
%In what follows, we present a new formulation of the cycle constraint projected in the space of variables $ W_{ij},~\forall (i,j) \in E$, i.e., without introducing artificial edges, and for any cycle dimension. 

\subsection{Cycle Constraints in the ${W}$ Space}

\begin{definition}
Given a graph $(N,E)$, a shortest undirected path starting at node
  $i \in N$ and reaching node $j \in N$ represents the minimal set of
  undirected edges linking $i$ to $j$ in $E \cup E^r$. 
\end{definition} 

In general, there can be several shortest paths linking two nodes. In this work,
without loss of generality, it suffices to consider one of these paths which we
denote $\mathcal P_{ij}$. For instance, consider the graph depicted in Figure \ref{graph},
we have $P_{11} = \emptyset,~\mathcal P_{12} = \{(1,2)\},~\mathcal
P_{13} = \{(1,2);(2,3)\},~ \mathcal P_{31} = \{(3,2);(2,1)\}$ and
$\mathcal P_{41} = \{(4,5);(5,1)\}$. We will use $\mathcal
N\left(\mathcal P\right)$ to denote the set of nodes in $\mathcal P$.

\begin{definition}
  The set of nodes separating $i$ from $j$ in $\mathcal P_{ij}$,
  denoted by $\mathcal S_{ij}$ is defined as
$$\mathcal S_{ij} = \mathcal N\left(\mathcal P_{ij}\right) \setminus \{i,j\}.$$
Moreover, given a cycle $\mathcal C$ containing node $i$, we define 
$$\mathcal S_i(\mathcal C) = \{i\} \cup \left(\bigcup\limits_{j \in \mathcal N\left(\mathcal C\right)} \mathcal S_{ij}\right)$$
where $\mathcal N\left(\mathcal C\right)$ denotes the set of nodes in cycle $\mathcal C$.
\end{definition}	

\begin{figure}[h!]
  \centering
    \includegraphics[scale=0.5]{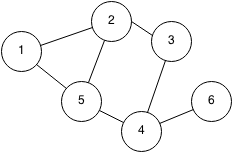}
  \caption{An Example Graph for Cycle Constraints. \label{graph}}
\end{figure}

For instance, consider the graph in \mbox{Figure \ref{graph}}
and the cycle $\mathcal C = \{(1,2);(2,3);(3,4);(4,5);(5,1)\}$. 
If we consider paths starting from node $1$, we have 
$\mathcal S_{11} = \mathcal S_{12} = \mathcal S_{15} =
\emptyset,~ \mathcal S_{13} = \{2\},~\text{and } \mathcal
S_{14} = \{5\}$. Therefore,  $\mathcal S_1(\mathcal C) = \{1,2,5\}$.
If we take node $2$ to be the source, we have $\mathcal S_{22} =
 \mathcal S_{21} = \mathcal S_{25} = \mathcal S_{23} =\emptyset,
 ~\text{and }  \mathcal S_{24} = \{3\}$. This leads to $\mathcal S_2(\mathcal C) = \{2,3\}$.

%\newpage

\begin{theorem}\label{th1}
  Given a cycle $\mathcal C$ and an arbitrary reference node $r$ in
  $\mathcal C$, Kirchhoff's Voltage Law can be expressed as
\begingroup
\small
\begin{equation}\label{loop}
\sum_{\substack{(i,j)\in \mathcal C}} \frac{\left(\prod\limits_{(k,l) \in \mathcal P_{rj}}W_{kl}^*\right) \left(\prod\limits_{s \in \mathcal S_r(\mathcal C)\setminus \mathcal N\left(\mathcal P_{rj}\right)}w_{s}\right)\left(W_{ij} - w_j\right)}{w_j} = 0
\end{equation}
\endgroup
\end{theorem}
\begin{proof}
\small
The proof uses the following two identities:
\begin{align}
&(V_i - V_j)V_j^* = W_{ij} - w_j \label{eq:id1}\\
&\frac{1}{V_j^*} = \frac{V_j}{w_j}\label{eq:id2}
\end{align}
%We have
\begingroup
\small
\begin{gather*}
\sum_{\substack{(i,j)\in \mathcal C}} (V_i - V_j) = 0\\
\Updownarrow\\
\sum_{\substack{(i,j)\in \mathcal C}} \frac{(V_i - V_j)V_j^*}{V_j^*} = 0\\
\Updownarrow \eqref{eq:id1}\\
\sum_{\substack{(i,j)\in \mathcal C}} \frac{W_{ij} - w_j}{V_j^*} = 0\\
\Updownarrow\eqref{eq:id2}\\
\sum_{\substack{(i,j)\in \mathcal C}} \frac{V_j(W_{ij} - w_j)}{w_j} = 0\\
\Updownarrow\\
V_{r}^*\left( \prod\limits_{k \in \mathcal S_r(\mathcal C)}V_k{V_k}^* \right)\sum_{\substack{(i,j)\in \mathcal C}} \frac{V_j(W_{ij} - w_j)}{w_j} = 0\\
\Updownarrow\\
\sum_{\substack{(i,j)\in \mathcal C}} \frac{V_{r}^*\left( \prod\limits_{k \in \mathcal S_r(\mathcal C)}V_k{V_k}^* \right)V_j(W_{ij} - w_j)}{w_j} = 0\\
\Updownarrow \eqref{eq:lift1},\eqref{eq:lift}\\
\sum_{\substack{(i,j)\in \mathcal C}} \frac{\left(\prod\limits_{(k,l) \in \mathcal P_{rj}}W_{kl}^*\right) \left(\prod\limits_{s \in \mathcal S_r(\mathcal C)\setminus \mathcal N\left(\mathcal P_{rj}\right)}w_{s}\right)(W_{ij} - w_j)}{w_j} = 0.
\end{gather*}
\endgroup
\end{proof}

\noindent
For illustration purposes, consider a three-bus cycle
$\mathcal C = \{(1,2),(2,3),(3,1)\}$. It is easy to check that $\mathcal P_{11} = \emptyset$, $\mathcal
P_{12} = \{(1,2)\}$, $\mathcal P_{13} = \{(1,3)\}$, leading to $\mathcal S_{ij} = \emptyset,~\forall (i,j) \in \mathcal C$, and $\mathcal S_1(\mathcal C) = \{1\}$.  Taking node
$1$ to be the reference node and applying Theorem \ref{th1}, we have,
\begin{gather*}
  \sum_{\substack{(i,j)\in \mathcal C}} \frac{\left(\prod\limits_{(k,l) \in \mathcal P_{1j}}W_{kl}^*\right) \left(\prod\limits_{s \in \mathcal S_1(\mathcal C)\setminus \mathcal N(\mathcal P_{1j})}w_{s}\right)(W_{ij} - w_j)}{w_j} = 0\\
\Downarrow\\
\frac{W_{12}^* (W_{12} - w_2)}{w_2} + \frac{W_{13}^* (W_{23} - w_3)}{w_3} + w_1\frac{W_{31} - w_1}{w_1} = 0\\
\Downarrow\\
\frac{|W_{12}|^2 - W_{12}^*w_2}{w_2} + \frac{W_{13}^* (W_{23} - w_3)}{w_3}  + W_{31} - w_1= 0\\
\Downarrow\\
\frac{w_{1}w_{2} - W_{12}^*w_2}{w_2} + \frac{W_{13}^* (W_{23} - w_3)}{w_3}  + W_{31} - w_1= 0\\
\Downarrow\\
w_{1} - W_{12}^* + \frac{W_{13}^*W_{23} }{w_3} -  W_{13}^* + W_{31} - w_1= 0\\
\Downarrow\\
W_{13}^{*}W_{23} = w_{3}W_{12}^{*}.
\end{gather*}
%\endgroup
In real number representation, this is equivalent to
\begin{align}
w^R_{31}w^R_{23} - w^I_{31}w^I_{23} &= w_3w^R_{12},\label{eq:cycle_r}\\
w^I_{31}w^R_{23} + w^R_{31}w^I_{23} &= -w_3w^I_{12}.\label{eq:cycle_i}
\end{align}
Note that one can generate another system of equations by setting the reference node to be $2$ before applying Theorem \ref{th1}, which leads to
\begin{align}
w^R_{31}w^R_{12} + w^I_{31}w^I_{12} &= w_1w^R_{23},\label{eq:cycle2_r}\\
w^I_{31}w^R_{12} - w^R_{31}w^I_{12} &= -w_1w^I_{23}.\label{eq:cycle2_i}
\end{align}
%Note that Kocuk et al. \cite{Kocuk_15} show that \eqref{eq:cycle_r}-\eqref{eq:cycle_i} is equivalent to the following constraint,
%\begin{equation}
%w_{12}^R\left(w_{23}^Rw_{13}^R + w_{23}^Iw_{13}^I\right) - w_{12}^I\left(w_{23}^Iw_{13}^R - w_{23}^Rw_{13}^I\right) = 0
%\end{equation}¥
\subsection{Linking the SDP Relaxation and Kirchhoff's Voltage Law}

Consider the $3 \times 3$ matrix defined in \eqref{hermit3}, the determinant constraint is written,
\begingroup
\small
\begin{gather*}
det({W} ) \ge 0\\
\Updownarrow\\
 - W_{12}\left(W_{12}^*W_{33} - W_{23}W_{13}^*\right) + W_{22} \left(W_{11}W_{33} - W_{13}W_{13}^*\right)\\
- W_{23}^*\left(W_{11}W_{23} - W_{13}W_{12}^*\right) \ge 0\\
\Updownarrow\\
 - |W_{12}|^2W_{33} + W_{12}W_{23}W_{13}^* + W_{22}W_{11}W_{33} - |W_{13}|^2W_{22}\\
- |W_{23}|^2W_{11} + \left(W_{23}W_{13}^*W_{12}\right)^* \ge 0\\
\Updownarrow\\
 2\mathcal R\left(W_{12}W_{23}W_{13}^*\right) + W_{11}W_{22}W_{33}\\
 \ge  |W_{12}|^2W_{33} + |W_{13}|^2W_{22} + |W_{23}|^2W_{11}\\
\Updownarrow\\
 2\left(w_{12}^R\left(w_{23}^Rw_{31}^R - w_{23}^Iw_{31}^I\right) - w_{12}^I\left(w_{23}^Iw_{31}^R + w_{23}^Rw_{31}^I\right)\right)\\
 \ge  |W_{12}|^2W_{33} + |W_{13}|^2W_{22} + |W_{23}|^2W_{11} - W_{11}W_{22}W_{33}.
% \ge  \left(\left(w_{12}^R\right)^2 + \left(w_{12}^I\right)^2\right)w_{3} + \left(\left(w_{13}^R\right)^2 + \left(w_{13}^I\right)^2\right)w_{2} +  \left(\left(w_{23}^R\right)^2 + \left(w_{23}^I\right)^2\right)w_{1}
\end{gather*}
\endgroup
Now, observe that the linear combinations of the cycle constraints
\begin{equation*}
 w^R_{12}\eqref{eq:cycle_r}+w^I_{12}\eqref{eq:cycle_i} \text{ and } w^R_{23}\eqref{eq:cycle2_r}+w^I_{23}\eqref{eq:cycle2_i}
\end{equation*}
lead to
\begingroup
\small
\begin{align}
w_{12}^R\left(w_{23}^Rw_{31}^R - w_{23}^Iw_{31}^I\right) - w_{12}^I\left(w_{23}^Iw_{31}^R + w_{23}^Rw_{31}^I\right) = |W_{12}|^2W_{33}\label{combi1}\\
w_{12}^R\left(w_{23}^Rw_{31}^R - w_{23}^Iw_{31}^I\right) - w_{12}^I\left(w_{23}^Iw_{31}^R + w_{23}^Rw_{31}^I\right) = |W_{23}|^2W_{11}\label{combi2}
\end{align}
\endgroup
which represent the main components of the SDP determinant constraint presented previousely. 
Furthermore, based on \eqref{eq:socp:det}, one has,
\begin{gather}
W_{11}W_{33} \ge |W_{13}|^2\nonumber\\
\Updownarrow\nonumber\\
W_{11}W_{22}W_{33} \ge W_{22}|W_{13}|^2\label{w123}
\end{gather}

\noindent
It is easy to show that \eqref{combi1} + \eqref{combi2} + \eqref{w123}
$\equiv$ \eqref{eq:det}.  This establishes the fact that, for
three-dimensional cycles, the SDP formulation captures a relaxed
version of the cycle constraints
\eqref{eq:cycle_r}-\eqref{eq:cycle2_i} combined with the SOCP
constraints \eqref{mat:socp}.

\section{Numerical Experiments}\label{sec:numerical}

This section presents a preliminary computational evaluation of the
polynomial SDP cuts.

In the current experiments, only three-dimensional SDP cuts were generated. Implementation to handle higher dimension matrices is ongoing work.

The relaxations were compared on a subset of the NESTA v0.5.0
\cite{nesta} benchmarks. NESTA is a comprehensive library including
state-of-the-art AC-OPF transmission system test cases ranging from 3
to 9000 nodes and consist of 35 different networks under three modes:
a typical operating condition (TYP), a congested operating condition
(API), small angle difference condition (SAD) and radial
configurations (RAD).  In our preliminary experiments, we focus on
small and medium size meshed instances of up to 300 nodes under
typical operating conditions.  Nonlinear models were solved using
IPOPT 3.12 \cite{Ipopt} with linear solver ma27 \cite{hsl_lib}.  The
SDP relaxation was executed on the state-of-the-art implementation
\cite{SDP} which already exploits the branch decomposition theorem
\cite{opfBranchDecomp}. The SDP solver SDPT3 4.0 \cite{Toh99sdpt3} was
used with the modifications suggested in \cite{SDP}. Note that the
phase angle bounds defined in Model \ref{mod:sdp} were also introduced
in the SDP formulation.  All instances were ran on a \amdq.

IPOPT \cite{Ipopt} is also used as a heuristic to find a feasible
solution to the AC-OPF problem, providing an upper bound value on the
optimal objective.  We then measure the {\em optimally gap} between
the heuristic and the relaxation using the formula
\begin{align}
\frac{\text{Heuristic - Relaxation}}{\text{Heuristic}} \nonumber
\end{align}

Table \ref{tbl:typ} and \ref{tbl:typ2} present the results in terms of
optimality gap and computational time, comparing the new polynomial
SDP cut formulation (P-SDP) with the original SDP model and the
standard SOCP relaxation respectively. Let us emphasize that we only
generate 3-dimensional determinant cuts in these experiments and that
the gap can be further reduced by generating cuts with higher
dimensions.  It is interesting to note that these 3-dimension cuts are
already reducing the optimality gap from $2.04\%$ to $0.68\%$ on
average, when compared to the SOCP relaxation. The gap reduction is
substantial on case\_30\_ieee, dropping from $15.88\%$ to
$0.06\%$. The computational time results are also very promising, the
new formulation is on average one order of magnitude faster than the
SDP solver approach.

\begin{table}[h!]
%\small
%\footnotesize
%\scriptsize
\small
\center
\caption{Comparing SDP with P-SDP on TYP instances}
\vspace{-0.2cm}
\begin{tabular}{|r||r|r||r|r|}
\hline
                  & \multicolumn{2}{c||}{Opt. Gap (\%)} & \multicolumn{2}{c|}{Runtime (seconds)} \\
Test Case & SDP & P-SDP & SDP & P-SDP \\
% sdpcuts_paper_bounds_time.txt
%%%%%%%%%%%%%%%%%%%%%%%%%%%%%%%%%%%%%
%Benchmark & AC & AC-SDP & AC-SPDC-QC & AC-SPDC & AC-SOC & AC & AC-SDP & AC-SDPC-QC & AC-SDPC & AC-SOC
\hline
\hline
  case3\_lmbd 	&	0.39	&	\bf 0.39	&	3.83	&	0.02	\\
 case4\_gs 	&	 0.00  	&	 \bf 0.00	&	4.00	&	0.03	\\
 case5\_pjm 	&	5.22	&	\bf 5.22	&	4.43	&	0.05	\\
 case6\_c 	&	0.00	&	\bf 0.00	&	4.31	&	0.05	\\
 case6\_ww 	&	 0.00  	&	0.62	&	4.54	&	0.04	\\
 case9\_wscc 	&	 0.00	&	\bf 0.00	&	3.95	&	0.20	\\
 case14\_ieee 	&	  0.00  	&	\bf 0.00	&	4.11	&	0.10	\\
 case24\_ieee\_rts 	&  0.00  	&	0.01	&	5.54	&	0.13	\\
 case29\_edin 	&	 0.00  	&	0.08	&	8.12	&	0.39	\\
 case30\_as 	&	0.00	&	\bf0.00	&	5.84	&	0.16	\\
 case30\_fsr 	&	0.00	&	0.37	&	6.04	&	0.15	\\
 case30\_ieee 	&  0.00  	&	0.06	&	5.56	&	0.26	\\
 case39\_epri 	&	0.01	&	 0.05$^\star$ 	&	6.90	&	0.29	\\
 case57\_ieee 	&	 0.00  	&	0.06	&	8.08	&	0.49	\\
 case73\_ieee\_rts 	&	0.00  	&	0.03	&	8.83	&	0.67	\\
 case89\_pegase 	&	0.00	&	0.17	&	18.79	&	2.44	\\
 case118\_ieee 	&	0.06	&	1.51	&	12.77	&	1.32	\\
 case162\_ieee\_dtc 	&	1.08	&	4.02	&	35.28	&	1.34	\\
 case189\_edin 	&	0.07	&	0.17	&	12.95	&	1.05	\\
 case300\_ieee 	&	0.08	&	0.93	&	27.70	&	2.66	\\
  \hline
 \bottomrule  
 Average 	&	0.34	&	0.68	&	9.57	&	0.59	\\
 \toprule
\end{tabular}\\
\vspace{0.3cm}
{\bf bold} - the polynomial cuts match the SDP gap,\\$\star$ - solver reported numerical accuracy warnings
\label{tbl:typ}
%\vspace{-0.3cm}
\end{table}

\begin{table}[h!]
%\small
%\footnotesize
%\scriptsize
\small
\center
\caption{Comparing P-SDP with SOCP on TYP instances}
\vspace{-0.2cm}
\begin{tabular}{|r||r|r||r|r|}
\hline
                  & \multicolumn{2}{c||}{Opt. Gap (\%)} & \multicolumn{2}{c|}{Runtime (seconds)} \\
Test Case & SOCP & P-SDP & SOCP & P-SDP \\
% sdpcuts_paper_bounds_time.txt
%%%%%%%%%%%%%%%%%%%%%%%%%%%%%%%%%%%%%
%Benchmark & AC & AC-SDP & AC-SPDC-QC & AC-SPDC & AC-SOC & AC & AC-SDP & AC-SDPC-QC & AC-SDPC & AC-SOC
\hline
\hline
nesta\_case3\_lmbd 	&	1.32	&	\bf0.39	&	 0.09  &	0.02\\
								
case4\_gs 	&	0	&	\bf0.00	&	 0.04  &	0.03\\
								
case5\_pjm 	&	14.54	&	\bf5.22	&	 0.06  &	0.05\\
								
case6\_c 	&	0.3	&	\bf0.00	&	 0.14  &	0.05\\
								
case6\_ww 	&	0.63	&	\bf0.62	&	 0.10  &	0.04\\
								
case9\_wscc 	&	0	&	\bf0.00	&	 0.05  &	0.2\\
								
case14\_ieee 	&	0.11	&	\bf0.00	&	 0.07  &	0.1\\
								
case24\_ieee\_rts 	&	0.01	&	\bf0.01	&	 0.08  &	0.13\\
								
case29\_edin 	&	0.12	&	\bf0.08	&	 0.18  &	0.39\\
								
case30\_as 	&	0.06	&	\bf0.00	&	 0.09  &	0.16\\
								
case30\_fsr 	&	0.39	&	\bf0.37	&	 0.08  &	0.15\\
								
case30\_ieee 	&	15.88	&	\bf0.06	&	 0.07  &	0.26\\
								
case39\_epri 	&	0.05	&	 \bf0.05$^\star$ 	&	 0.15  &	0.29\\
								
case57\_ieee 	&	0.06	&	\bf0.06	&	 0.14  &	0.49\\
								
case73\_ieee\_rts 	&	0.03	&	\bf0.03	&	 0.24  &	0.67\\
								
case89\_pegase 	&	0.17	&	\bf0.17	&	 0.34  &	2.44\\
								
case118\_ieee 	&	1.83	&	\bf1.51	&	 0.20  &	1.32\\
								
case162\_ieee\_dtc 	&	4.03	&	\bf4.02	&	 0.30  &	1.34\\
								
case189\_edin 	&	0.21	&	\bf0.17	&	 0.37  &	1.05\\
								
case300\_ieee 	&	1.18	&	\bf0.93	&	 0.50  &	2.66\\
  \hline
 \bottomrule  
 Average 	&	2.04	&	\bf 0.68	&	0.16	&	0.59	\\
 \toprule
\end{tabular}\\
\vspace{0.3cm}
{\bf bold} - smaller optimality gap,\\$\star$ - solver reported numerical accuracy warnings
\label{tbl:typ2}
%\vspace{-0.3cm}
\end{table}

\section{Conclusion} 
\label{sec:conc}

The polynomial SDP cuts introduced in this paper have a great
potential in increasing the scalability of semidefinite programming
approaches to tackle complex optimisation problems in power systems.
The computational time improvement can have a huge impact when the
model includes discrete variables and branching becomes mandatory. The
time reduction then becomes a factor of the number of nodes explored
in the branch \& bound tree.

Further implementation work is needed to automatically generate the
SDP determinant cuts for matrices with dimensions higher than three.
The key idea would be to generate the cuts lazily using a separation
algorithm. Such an approach would start by solving the SOCP
relaxation, identifying a submatrix violating the determinant
constraint, and adding the corresponding cut to the model. The process
is then iterated until the duality gap is small enough or the matrix
is shown to be positive semidefinite.

% or


\section*{Acknowledgment}
This work was conducted at NICTA and is funded by the
  Australian Government as represented by the Department of Broadband,
  Communications and the Digital Economy and the Australian Research
  Council through the ICT Centre of Excellence program.

\ifCLASSOPTIONcaptionsoff
  \newpage
\fi



%\bibliographystyle{IEEEtran}
% \bibliography{SDP_cuts}


%\thanksto{}
\end{document}